\documentclass[11pt,a4paper]{article}

\usepackage{geometry}                
\usepackage{amsmath}
\usepackage{amsthm}
\usepackage{amssymb}
\usepackage{placeins}
\usepackage{enumitem}
\usepackage{graphicx}
\usepackage{ amssymb }
\usepackage{epstopdf}
\usepackage[utf8]{inputenc}
\usepackage[english]{babel}
\usepackage{amsthm}
\usepackage{ dsfont }
\usepackage[square,sort,comma,numbers]{natbib}
\usepackage{amsmath}
\usepackage{amsfonts}
\usepackage{amssymb}
\usepackage{scalerel}[2016/12/29]
\usepackage{bbm}
\usepackage{amsmath,amsfonts,amsthm,amssymb}
\usepackage{fancyhdr}
\usepackage{float}
\usepackage{ amssymb }

\theoremstyle{plain}
\newtheorem{theorem}{Theorem}[section]
\newtheorem{corollary}[theorem]{Corollary}
\newtheorem{prop}[theorem]{Proposition}
\newtheorem{lemma}[theorem]{Lemma}

\theoremstyle{definition}

\newtheorem{remark}[theorem]{Remark}
\newtheorem{conj}[theorem]{Conjecture}  

\theoremstyle{plain}
\newtheorem{innercustomthm}{Theorem}
\newenvironment{customthm}[1]
  {\renewcommand\theinnercustomthm{#1}\innercustomthm}
  {\endinnercustomthm}

\newcommand{\inte}{\mathbb{Z}}
\newcommand{\rat}{\mathbb{Q}}
\newcommand{\nat}{\mathbb{N}}
\newcommand{\real}{\mathbb{R}}

\newcommand{\prob}{\mathbb{P}}
\newcommand{\expt}{\mathbb{E}}
\newcommand{\indic}{\mathbbm{1}}

\newcommand{\floor}[1]{{\left\lfloor #1 \right\rfloor}}
\newcommand{\ceil}[1]{{\left\lceil #1 \right\rceil}}

\newcommand{\sset}{\subset}

\newcommand{\al}{\alpha}
\newcommand{\Om}{\Omega}
\newcommand{\mathforall}{\text{ for all }}

\newcommand{\mathand}{\;\text{and}\;}

\newcommand{\mathas}{\;\text{as}\;}
\newcommand{\mathsuchthat}{\;\text{such that}\;}

\newcommand{\Ga}{\Gamma}
\newcommand{\ep}{\epsilon}

\newcommand{\de}{\delta}
\newcommand{\De}{\Delta}
\newcommand{\sig}{\sigma}

\newcommand{\scrA}{\mathcal{A}}

\newcommand{\scrD}{\mathcal{D}}

\newcommand{\scrM}{\mathcal{M}}

\newcommand{\scrF}{\mathcal{F}}

\newcommand{\card}[1]{\left\vert #1 \right\vert}

\newcommand{\supp}{\text{supp}}
\newcommand{\Z}{\mathds{Z}}
\newcommand{\ddd}{\mathellipsis}
\usepackage{MnSymbol}

\newcommand{\Var}{\text{\fontfamily{ppl}\selectfont Var}}

\usepackage{yfonts}
\usepackage{ wasysym }

\newcommand{\eqd}{\stackrel{d}{=}}

\newcommand{\cvgd}{\stackrel{d}{\to}}
\newcommand{\cvgp}{\stackrel{\prob}{\to}}
\newcommand{\X}{\times}

\newcommand{\rev}{\text{rev}}
\newcommand{\id}{\text{id}}
\newcommand{\as}{\text{almost surely}}

\newcommand{\lf}{\left}
\newcommand{\rg}{\right}
\usepackage{titling}

\author{Duncan Dauvergne and B\'alint Vir\'ag}                
\title{Circular support in random sorting networks}
\begin{document}
\maketitle

\FloatBarrier
\begin{abstract}
A sorting network is a shortest path from $12 \cdots n$ to $n \cdots 2 1$ in the Cayley graph of the symmetric group generated by adjacent transpositions. For a uniform random sorting network, we prove that in the global limit, particle trajectories are supported on $\pi$-Lipschitz paths. We show that the weak limit of the permutation matrix of a random sorting network at any fixed time is supported within a particular ellipse. This is conjectured to be an optimal bound on the support. We also show that in the global limit, trajectories of particles that start within distance $\ep$ of the edge are within $\sqrt{2\ep}$ of a sine curve in uniform norm. 
\end{abstract}

\begin{figure}[H]
   \centering
   \includegraphics[scale= 1]{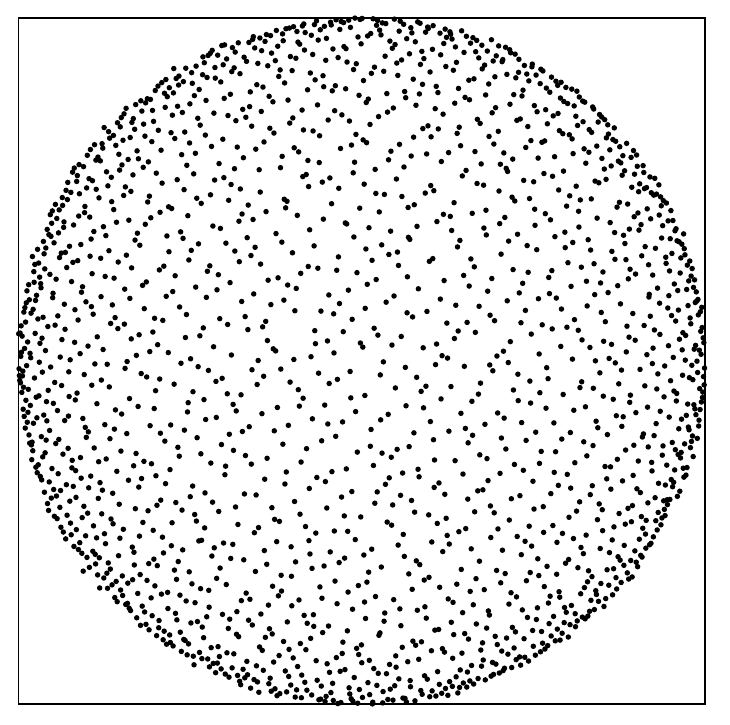}
   \caption{The permutation matrix of the half-time permutation for a $2000$ element sorting network. We prove that in the weak limit, the support of the half-time permutation matrix lies inside the unit disk. This figure originally appeared in \cite{angel2007random}.}
   \label{fig:halfway}
\end{figure}

\section{Introduction}

Consider the Cayley graph $\Ga(S_n)$ of the symmetric group $S_n$ with generators given by adjacent transpositions $\pi_i = (i, i + 1), i \in \{1, \ddd n -1\}$. A {\bf sorting network} is a minimal length path in $\Ga(S_n)$ from the identity permutation $\id_n = 1 2 \cdots n$ to the reverse permutation $\rev_n = n \cdots 2 1$. The length of such paths is $N = {n \choose 2}$. 

\medskip

Sorting networks are also known as {\bf reduced decompositions} of the reverse permutation, as any sorting network can equivalently be represented as a minimal length decomposition of the reverse permutation as a product of adjacent transpositions: $\rev_n = \pi_{k_N} \ddd \pi_{k_1}$. In this setting, the path in the Cayley graph is the sequence 
$$
\lf\{ \pi_{k_i} \cdots \pi_{k_2} \pi_{k_1}: i \in \lf\{0, \ddd N \rg\} \rg\} .
$$
The combinatorics of sorting networks have been studied in detail under this name. There are connections between sorting networks and Schubert calculus, quasisymmetric functions, zonotopal tilings of polygons, and aspects of representation theory. For more background in this direction, see Stanley \cite{stanley1984number}; Bjorner and Brenti \cite{bjorner2006combinatorics}; Garsia \cite{garsia2002saga}; Tenner \cite{tenner2006reduced}; and Manivel \cite{manivel2001symmetric}.
%
%
%
%
%

\medskip

In computer science, sorting networks are viewed as $N$-step algorithms for sorting a list of $n$ numbers. At step $i$ of the algorithm, we sort the elements at positions $k_i$ and $k_i + 1$ into increasing order. This process sorts any list in $N$ steps.

\medskip

In order to understand the geometry of sorting networks, we think of the numbers $\{1, \ddd, n \}$ as labelled particles being sorted in time (see Figure \ref{fig:wiring}).
We use the notation $\sig(x, t) = \pi_{k_\floor{t}} \ddd \pi_{k_2}\pi_{k_1}(x)$ for the position of particle $x$ at time $t$. 

\begin{figure}
   \centering
   \includegraphics[scale=0.7]{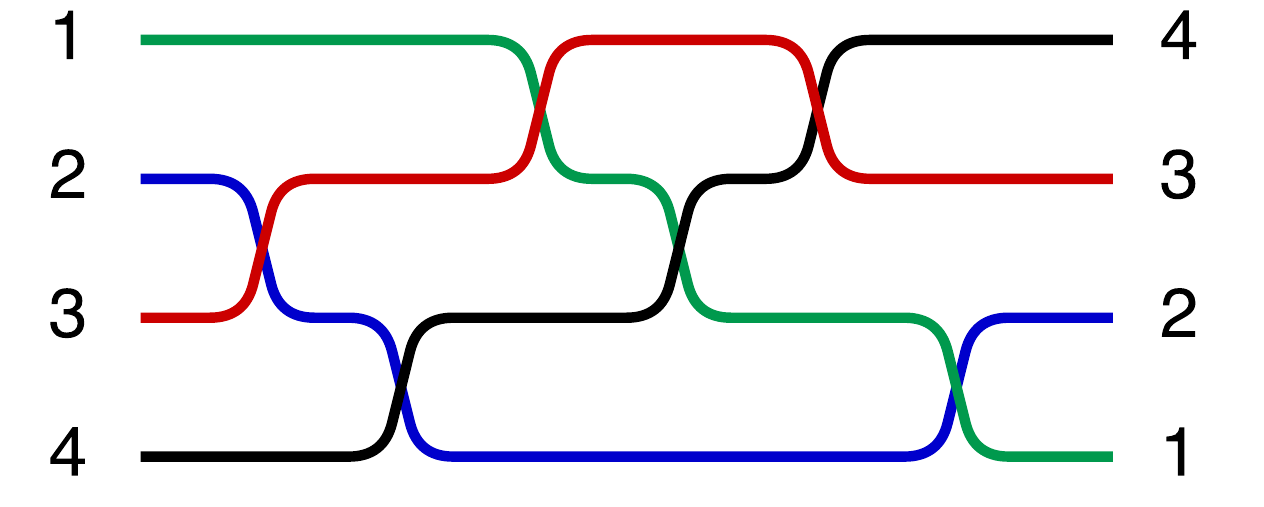}
   \caption{A ``wiring diagram" for a sorting network with $n = 4$. In
     this diagram, trajectories are drawn as continuous curves for
     clarity.}
   \label{fig:wiring}
\end{figure}

\medskip

Angel, Holroyd, Romik, and Vir\'ag \cite{angel2007random} initiated the study of uniform random sorting networks. Based on numerical evidence, they made striking conjectures about their global behaviour.
\medskip

 Their first conjecture concerns the rescaled trajectories of a uniform random sorting network. In this rescaling, space is scaled by $2/n$ and shifted so that particles are located in the interval $[-1, 1]$. Time is scaled by $1/N$ so that the sorting process finishes at time $1$. Specifically, we define the {\bf global trajectory} of particle $x$ by 
$$
\sig_G(x, t) = \frac{2\sig(x, Nt)}n - 1.
$$ 
In \cite{angel2007random}, the authors conjectured that global trajectories converge to sine curves (see Figure \ref{fig:sinecurves}). They proved that limiting trajectories are H\"older-$1/2$ continuous with H\"older constant $\sqrt{8}$. To precisely state their conjecture, we use the notation $\sig^n$ for an $n$-element uniform random sorting network.
\begin{conj}
\label{CJ:sine-curves}
For each $n$ there exist random variables $\{(A^n_x, \Theta^n_x) : x = 1, \ddd n \}$ such that for any $\ep > 0$, 
$$
\prob \lf( \max_{x \in [1, n] } \sup_{t \in [0, 1]} \card{ \sig^n_G(x, t) - A_x^n \sin(\pi t + \Theta_x^n) } > \ep \rg) \to 0 \qquad \mathas n \to \infty.
$$
\end{conj}

\begin{figure}
   \centering
   \includegraphics[scale=0.8]{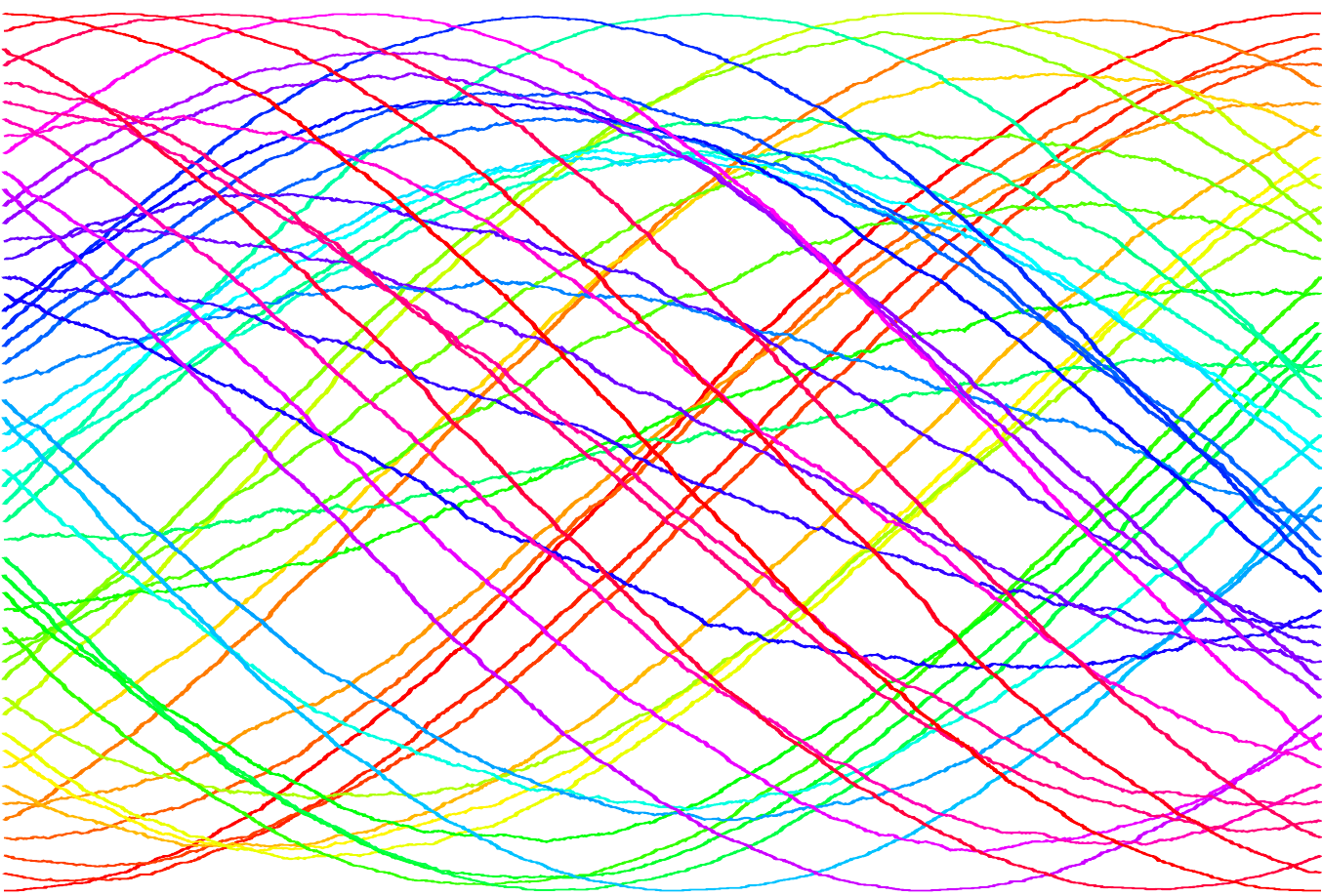}
   \caption{A diagram of selected particle trajectories in a 2000 element sorting network. This image is taken from \cite{angel2007random}.}
   \label{fig:sinecurves}
\end{figure}

Their second conjecture concerns the time-$t$ permutation matrices of a uniform sorting network. First, let the Archimedean measure $\mathfrak{Arch}_{1/2}$ on the square $[-1, 1]^2$ be the probability measure with Lebesgue density
$$
f(x, y) = \frac{1}{2\pi \sqrt{1 - x^2 - y^2}}
$$
on the unit disk, and $0$ outside. The measure $\mathfrak{Arch}_{1/2}$ is the projected surface area measure of the $2$-sphere.
For general $t$, define $\mathfrak{Arch}_t$ to be the distribution of
$$
(X, X\cos(\pi t) + Y \sin (\pi t)), \qquad \text{where } (X, Y) \sim \mathfrak{Arch}_{1/2}.
$$ 
In \cite{angel2007random}, the authors conjectured that the time-$t$ permutation matrix of a uniform sorting network converges to $\mathfrak{Arch}_t$ (see Figure \ref{fig:circles}). They proved that for any $t$, the support of the time-$t$ permutation matrix is contained in a particular octagon with high probability.
\begin{conj}
\label{CJ:matrices}
Consider the random measures 
\begin{equation}
\label{E:eta-n-t}
\eta^n_t = \frac{1}n \sum_{i = 1}^n \delta(\sig^n_G(i, 0), \sig^n_G(i, t)).
\end{equation}
Here $\delta(x, y)$ is a $\delta$-mass at $(x, y)$. Then for any $t \in [0, 1]$,
$$
\eta^n_t \to \mathfrak{Arch}_t \qquad \text{in probability in the weak topology}.
$$
That is, for any weakly open neighbourhood $O$ of $\mathfrak{Arch}_t$, $\prob(\eta^n_t \in O) \to 1$ as $n \to \infty.$
\end{conj}

\begin{figure}
   \centering
   \includegraphics[scale=0.9]{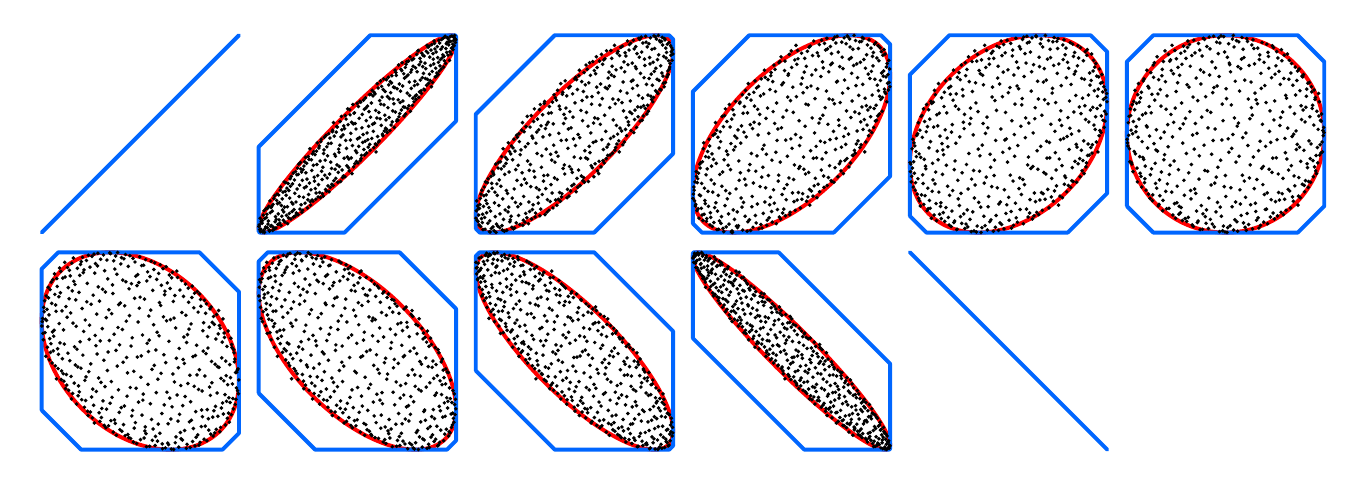}
   \caption{A diagram of the measures $\{\eta^n_t : t \in \{0, 1/10, 2/10, \ddd 1\} \}$ in a $500$-element sorting network. The octagon bounds from \cite{angel2007random} are given in blue in this picture. One of our main results in this paper is proving the ellipse bounds (red). As can be seen from the figure, simulations suggest that this bound is tight. This figure is from \cite{angel2007random}.}
   \label{fig:circles}
\end{figure}
%
%

\medskip

The main results of this paper work towards proving the above two conjectures. To state these results, let $\scrD$ be the closure of the space of all possible sorting network trajectories under the uniform norm. Let $Y_n \in \scrD$ be a unifomly chosen particle trajectory from the set of $n$-element sorting network trajectories. That is, if $\sig^n$ is a uniform $n$-element sorting network, and $I_n$ is an independent uniform random variable on $\{1, \ddd n\}$, then
 $$
 Y_n = \sig^n_G(I_n, \cdot).
 $$
The following lemma, proven in Section \ref{S:prelim}, guarantees that subsequential limits of $Y_n$ exist in distribution. This is a version of the H\"older continuity result from \cite{angel2007random}.

\begin{lemma}
\label{L:precompact}

(i) The sequence $\{Y_n : n \in \nat\}$ is uniformly tight.

(ii) Let $Y$ be a subsequential limit of $\{Y_n : n \in \nat\}$ in distribution. Then
$$
\prob \lf( Y \text{ is H\"older-$1/2$ continuous with H\"older constant $\sqrt{8}$} \mathand Y(0) = -Y(1) \rg) =1.
$$
Moreover, for each $t \in [-1, 1]$, $Y(t)$ is uniformly distributed on $[-1, 1]$.

\end{lemma}

\medskip

We say that a path $y \in \scrD$ is $g(y)$-Lipschitz if $y$ is absolutely continuous and if for almost every $t$, $|y'(t)| \le g(y(t))$. We can now state the main theorem of this paper.

\begin{theorem}
\label{T:main}
Suppose that $Y$ is a distributional subsequential limit of $Y_n$. Then 
$$
\prob \lf(Y \text{ is } \pi\sqrt{1 - y^2}\text{-Lipschitz}\rg) = 1.
$$
\end{theorem}

As a consequence of Theorem \ref{T:main}, we show that any weak limit of the time-$t$ permutation matrices is contained in the elliptical support of $\mathfrak{Arch}_t$. We also show that trajectories near the top of sorting networks are close to sine curves. 

\begin{theorem}
\label{T:main-3}
Let $t \in [0, 1]$, and let $\eta_t$ be a subsequential limit of $\eta^n_t$. Then the support of the random measure $\eta_t$ is almost surely contained in the support of $\mathfrak{Arch}_t$.
\end{theorem}

\begin{theorem}
\label{T:main-2}
Suppose that $Y$ is a subsequential limit of $Y_n$. Then for any $\ep > 0$,
\begin{align*}
\prob \lf( Y(0) \ge 1 - \ep \mathand ||Y(t) - \cos(\pi t)||_u \ge \sqrt{2\ep}\rg) &= 0, \qquad \mathand \\
\prob \lf(Y(0) \le -1 + \ep \mathand ||Y(t) + \cos(\pi t)||_u \ge \sqrt{2\ep}\rg) &= 0.
\end{align*}
Here $||\cdot||_u$ is the uniform norm.
\end{theorem}

\subsection{Local limit theorems}

In order to prove Theorem \ref{T:main}, we analyze the interactions between the local and global structure of sorting networks. As a by-product of this analysis, we prove that in the local limit of random sorting networks, particles have bounded speeds and swap rates. To state these theorems, we first give an informal description of the local limit (a precise description is given in Section \ref{S:prelim}). The existence of this limit was established independently by Angel, Dauvergne, Holroyd, and Vir\'ag \cite{angel2017local}, and by Gorin and Rahman \cite{gorin2017}. Define the local scaling of trajectories
$$
U_n (x, t) = \sig^n(\floor{n/2} + x, nt) - \floor{n/2}.
$$
With an appropriate notion of convergence, we have that 
$$
U_n \cvgd U,
$$
where $U$ is a random function from $\Z \X [0, \infty) \to \Z$. $U$ is the local limit centred at particle $\floor{n/2}$. We can also take a local limit centred at particle $\floor{\al n}$ for any $\al \in (0, 1)$. The result is the process $U$ with time rescaled by a semicircle factor $2\sqrt{\al(1 - \al)}$. We now state our two main theorems about $U$. 
\begin{theorem}
 \label{T:main-local}
For every $x \in \Z$, the following limit
$$
S(x) = \lim_{t \to \infty} \frac{U(x, t) - U(x, 0)}{t}  \qquad \text{exists } \as.
$$
$S(x)$ is a symmetric random variable with distribution $\mu$ independent of $x$. The support of $\mu$ is contained in the interval $[-\pi, \pi]$.
Moreover, the random function $S: \Z \to \real$ is stationary and mixing of all orders with respect to the spatial shift $\tau$ given by  $\tau S(x) = S(x + 1)$.
\end{theorem}

We call $\mu$ the {\bf local speed distribution}. Theorem \ref{T:main-local} is proven as Corollary \ref{C:speed-exist} and Theorem \ref{T:bounded-speed}. 
To state the second theorem, let $Q(x, t)$ be the number of swaps made by particle $x$ in the interval $[0, t]$.

\begin{theorem}
\label{T:swap-rate}
Let $x \in \Z$, and let $S(x)$ be as in Theorem \ref{T:main-local}. Then
$$
\lim_{t \to \infty} \frac{Q(x, t)}t = \int |y - S(x)|d\mu(y) \qquad \text{almost surely and in $L^1$.}
$$
\end{theorem}

Note that the speed distribution $\mu$ is not supported on a single point, so the process $U$ is not ergodic in time. In fact, Corollary \ref{C:av-swap-rate} shows that if $X$ and $X'$ are two independent samples from $\mu$, then $\expt |X - X' | = 8/\pi$.

\subsection*{Further Work}
In a subsequent paper \cite{dauvergne3}, the first author uses the results of this paper as a starting point for proving all the sorting network conjectures from \cite{angel2007random}. In particular, this proves Conjectures \ref{CJ:sine-curves} and \ref{CJ:matrices}.

\subsection*{Related Work}
Different aspects of random sorting networks have been studied by Angel and Holroyd \cite{angel2010random}; Angel, Gorin and Holroyd \cite{angel2012pattern}; Reiner \cite{reiner2005note}; Tenner \cite{tenner2014expected}; and Fulman and Stein \cite{fulman2014stein}. In much of the previous work on sorting networks, the main tool is a bijection of Edelman and Greene \cite{edelman1987balanced} between Young tableaux of shape $(n-1, n-2, \ddd 1)$ and sorting networks of size $n$. Little \cite{little2003combinatorial} found another bijection between these two sets, and Hamaker and Young \cite{HY} proved that these bijections coincide.

\medskip

Interestingly in our work and in the subsequent work \cite{dauvergne3}, we are able to work purely with previous known results about sorting networks and avoid direct use of the combinatorics of Young tableaux. As mentioned above, our starting point is the local limit of random sorting networks \cite{angel2017local, gorin2017}, though interestingly we only use a few probabilistic facts about this limit and never use any of the determinantal structure proved in \cite{gorin2017}. Other than basic sorting network symmetries, the only other previously known results that enter into our proofs and those in \cite{dauvergne3} are a bound on the longest increasing subsequence in a random sorting network from \cite{angel2007random} and consequences of this bound (H\"older continuity and the permutation matrix 'octagon' bound). 

\medskip

 Problems involving limits of sorting networks under a different measure and with different restrictions on the length of the path in $\Ga(S_n)$ have been considered by Angel, Holroyd, and Romik \cite{angel2009oriented}; Kotowski and Vir\'ag \cite{kotowski2016limits}; Rahman, Vir\'ag, and Vizer \cite{rahman2016geometry}; and Young \cite{young2014markov}. 
 
 \medskip
 
 In particular, in \cite{kotowski2016limits} (see also \cite{rahman2016geometry}), the authors prove that trajectories in reduced decompositions of $\rev_n$ of length $n^{2 + \ep}$ for some $\ep \in (0, 1)$ converge to sine curves, proving the `relaxed' analogues of Conjectures \ref{CJ:sine-curves} and \ref{CJ:matrices}. They do this by using large deviations techniques from the field of interacting particle systems. However, it appears to be very difficult to say anything about random sorting networks using this approach. Instead, both this paper and the subsequent work \cite{dauvergne3} take an entirely different approach based around patching together local swap rate information to deduce global structure.

\subsection*{Overview of the proofs and structure of the paper}

The guiding principle behind our proofs is that we can gain insight into both the local and global structure of random sorting networks by thinking of a large-$n$ sorting network as consisting of many local limit-like blocks. By doing this, we can show that if the local limit behaves too badly, then this contradicts a global bound, and similarly if the local limit behaves well, then this forces global structure.

\medskip

We first show that particle speeds exist and are bounded in the local limit. The existence of particle speeds follows from stationarity properties of the local limit, and is proven in Section \ref{S:existence}.
To show that speeds are bounded, we connect the local and global structure of sorting networks. If the local speed distribution is not supported in $[-\pi, \pi]$, then spatial ergodicity of the local limit guarantees that there are particles travelling with local speed greater than $\pi$ in most places in a typical large-$n$ sorting network $\sig$. By patching together the movements of these fast particles, we can create a long increasing subsequence in the swap sequence for $\sig$. This contradicts a theorem from \cite{angel2007random} and finishes the proof of Theorem \ref{T:main-local}. This is done in Section \ref{S:bounded-speed}.

\medskip

In Section \ref{S:local-swap-rates} and \ref{S:Lipschitz}, we complete the proof of Theorem \ref{T:main} by showing that control over the local speed of particles gives us control over their global speeds. By the bound on local speeds, most particles in a typical large-$n$ sorting network move with local speed in $[-\pi, \pi]$ most of the time. To control what happens when particles don't move with speeds in this range, we first prove a lower bound on the number of swaps that occur when particles do move with speed in $[-\pi, \pi]$ (essentially Theorem \ref{T:swap-rate}). This shows that not too many swaps, and hence not too much particle movement, can occur when particle speeds are not in this range. 

%

\medskip

Theorem \ref{T:main-2} and Theorem \ref{T:main-3} follow easily from Theorem \ref{T:main} and are proven in Section \ref{S:corollaries}. In particular, the fact that edge trajectories are close to sine curves is due to the fact that for a particle starting near the edge to reach its destination along a $\pi\sqrt{1-y^2}$-Lipschitz trajectory, it must move with speed close to $\pi$ most of the time.

\section{Preliminaries}
\label{S:prelim}
In this section we collect necessary facts about sorting networks, and recall a precise definition of the local limit. We also prove Lemma \ref{L:precompact}.

\medskip

A basic fact about sorting networks is that they exhibit time-stationarity. Specifically, we have the following theorem, first observed in \cite{angel2007random}.

\begin{theorem}
\label{T:basic-time-stat}
Let $(K_1, \ddd K_N)$ be the swap sequence of an $n$-element uniform random  sorting network $\sig^n$. We have that
$$
(K_1, \ddd K_N) \eqd (K_2, \ddd K_N, n + 1 - K_1).
$$
\end{theorem}
\noindent This theorem follows from the observation that the map
$$
\{k_1, \ddd k_N\} \mapsto  \{k_2, \ddd k_N, n + 1 - k_1\}
$$
is a bijection in the space of $n$-element sorting network swap sequences. The second theorem that we need bounds the length of the longest increasing subsequence in an initial segment of the swap sequence for a random sorting network. This result is proven in \cite{angel2007random} as Corollary 15 and Lemma 18 (though it is not written down formally as a theorem itself).

\begin{theorem}
\label{T:subsequence}
Let $L_n(t)$ be the length of the longest increasing subsequence of $(K_1, K_2, \ddd K_\ceil{Nt})$. 
Then for any $\ep > 0$, we have that
$$
\prob \lf( \max_{t \in [0, 1]} \big|L_n(t) - n\sqrt{t(2 - t)}\big| > \ep n\rg) \to 0 \qquad \mathas n \to \infty.
$$

\end{theorem}

We also need the result regarding H\"older continuity of trajectories from \cite{angel2007random}.

\begin{theorem}
\label{T:holder}
For any $\ep > 0$, the global particles trajectories of $\sig^n$ satisfy
$$
\lim_{n \to \infty} \prob \lf( |\sig^n_G(x, t) - \sig^n_G(x, s)| \le \sqrt8|t - s|^{1/2} + \ep \mathforall x \in [1, n], s, t \in [0, 1] \rg) = 1.
$$
\end{theorem}

Theorem \ref{T:holder} can be used to immediately prove Lemma \ref{L:precompact}. Recall that $Y_n$ is the trajectory random variable on $n$-element sorting networks.

\begin{proof}[Proof of Lemma \ref{L:precompact}.]

Let
$$
A_\ep = \lf\{ f \in \scrD : |f(t) - f(s)| \le \sqrt{8}|t -s|^{1/2} + \ep \mathforall s, t \in [0, 1]\rg\}.
$$
By Theorem \ref{T:holder}, we can find a sequence $\ep_n \to 0$ such that
\begin{equation}
\label{E:A-ep-n}
\lim_{n \to \infty} \prob \lf(Y_n \in A_{\ep_n} \rg) = 1.
\end{equation}
For a function $f \in \scrD$, define the $m$th linearization $f_m$ of $f$ by letting $f_m(i/m) = f(i/m)$ for all $i \in \{0, \ddd m \}$, and by setting $f_m$ to be linear at times in between. 

\medskip

Now fix $\de > 0$. There exists a sequence $m_\de(n) \to \infty$ as $n \to \infty$ such that for large enough $n$, if $f \in A_{\ep_n}$, then $f_{m_\de(n)}$ is H\"older-$1/2$ continuous with H\"older constant $\sqrt{8} + \de$. Moreover, 
there exists a sequence $c_n \to 0$ such that if $f \in A_{\ep_n}$, then the uniform norm 
\begin{equation}
\label{E:error-lin}
||f - f_{m_\de(n)}||_u \le c_n.
\end{equation}
For each $n$, define the random variable $Y_{n, m}$ to be the $m$th linearization of $Y_n$. By \eqref{E:A-ep-n} and \eqref{E:error-lin}, a subsequence $Y_{n_i} \to Y$ in distribution if and only if $Y_{n_i, m_\de(n_i)} \to Y$ in distribution. Moreover, \eqref{E:A-ep-n}  implies that the probability that $Y_{n_i, m_\de(n_i)}$ is H\"older-$1/2$ continuous with H\"older constant $\sqrt{8} + \de$ approaches $1$ as $n \to \infty$.

 Therefore both $Y_{n, m_\de(n)}$ and $Y_n$ are uniformly tight, and any subsequential limit $Y$ of $Y_n$ must be supported on the set of H\"older-$1/2$ continuous functions with H\"older constant $\sqrt{8} + \de$.

\medskip

This holds for all $\de > 0$, giving the H\"older continuity in the statement of the lemma. The rest of part (ii) of Lemma \ref{L:precompact} follows directly from the definition of $Y_n$.
\end{proof}

\begin{remark}
For a sorting network $\sig$, let $\nu_\sig$ be uniform measure on the trajectories $\{\sig_G(i, \cdot)\}_{i \in \{1, \dots, n\}}$. Letting $\Om_n$ be the space of all $n$-element sorting networks, consider the random measure
$$
\nu_n = \frac{1}{\card{\Om_n}} \sum_{\sig \in \Om_n} \nu_\sig.
$$
Let $\scrM(\scrD)$ be the space of probability measures on $\scrD$ with the topology of weak convergence, and let $\scrM(\scrM(\scrD))$ be the space of probability measures on $\scrM(\scrD)$ with the topology of weak convergence.

Essentially the same proof as that of Lemma \ref{L:precompact} can be used to show that the sequence $\{\nu_n\}_{n \in \nat}$ is precompact in $\scrM(\scrM(\scrD))$. This is stronger than the statement that the sequence $\{Y_n\}_{n \in \nat}$ is precompact, since the law of $Y_n$ can be thought of as the expectation of $\nu_n$.

Theorems \ref{T:main}, \ref{T:main-3}, and \ref{T:main-2} can also all be stated for subsequential limits of $\nu_n$.
\end{remark}

\subsection{The local limit}
Define a {\bf swap function} as a map $U:\inte \times [0, \infty) \to \Z$ satisfying the following properties:

\smallskip

 \begin{enumerate}[nosep,label=(\roman*)]
  \item For each $x$, $U(x,\cdot)$ is cadlag with nearest neighbour jumps.
  \item For each $t$,  $U(\cdot,t)$ is a bijection from $\Z$  to $\Z$.
  \item Define $U^{-1}(x, t)$ by $U(U^{-1}(x, t),t) = x$.  Then for each $x$, $U^{-1}(x, \cdot)$
    is a cadlag path with nearest neighbour jumps.
  \item For any time $t \in (0, \infty)$ and any $x \in \Z$, 
  $$
  \lim_{s \to t^-} U^{-1}(x, s) = U^{-1}(x + 1, t) \qquad \text{if and only if} \qquad  \lim_{s \to t^-} U^{-1}(x +1, s) = U^{-1}(x, t).
  $$
  \end{enumerate}
  We think of a swap function as a collection of particle trajectories $\{U(x, \cdot) : x \in \Z\}$.
Condition (iv) guarantees that the only way that a particle at position $x$ can move up at time $t$ is if the particle at position $x+1$ moves down. That is, particles move by swapping with their neighbours.   

\medskip
  
  Let $\scrA$ be the space of swap functions endowed with the following topology. A sequence of swap functions $U_n \to U$ if each of the cadlag paths $U_n(x, \cdot) \to U(x, \cdot)$ and $U^{-1}_n(x, \cdot) \to U^{-1}(x, \cdot)$. Convergence of cadlag paths is convergence in the Skorokhod topology. We refer to a random swap function as a \textbf{swap process}.
  
  \medskip
  
For a swap function $U$ and a time $t \in (0, \infty)$, define
$$
U(\cdot, t, s) = U(U^{-1}(\cdot, t), t + s).
$$
The function $U(\cdot, t, s)$ is the increment of $U$ from time $t$ to time $t + s$. 

\medskip

Now let $\al \in (-1, 1)$, and let $\{k_n \}_{n \in \nat}$ be any sequence of integers such that $k_n/n \to (1 + \al)/2$. Consider the shifted, time-scaled swap process 
$$
U_{n}^{k_n}(x, s) = \sig^n \lf(k_n + x, \frac{ns}{\sqrt{1- \al^2}} \rg) - k_n.
$$
To ensure that $U_{n}^{k_n}$ fits the definition of a swap process, we can extend it to a random function from $\Z \X [0, \infty) \to \Z$ by letting $U_{n}^{k_n}$  be constant after time $\frac{n-1}{2\sqrt{1 - \al^2}}$, and with the convention that  $U_{n}^{k_n}(x, s)= x$ whenever $x \notin \{1 -k_n, \ddd n - k_n\}$. In the swap processes $U_{n}^{k_n}$, all particles are labelled by their initial positions.
The following is shown in \cite{angel2017local}, and also essentially in \cite{gorin2017}. 

\begin{theorem}
\label{T:local}
There exists a swap process $U$ such that for any $\al, k_n$ satisfying the above conditions,
  $$
U_{n}^{k_n}  \cvgd U \qquad \mathas \; n \to \infty.
  $$
The swap process $U$ has the following properties:
\begin{enumerate}[nosep,label=(\roman*)]
 \item $U$ is stationary and mixing of all orders with respect to the spatial shift $\tau U(x, t) = U(x + 1, t) - 1$.
  \item $U$ has stationary increments in time: for any $t \ge 0$, the
  process $U(\cdot, t, s)_{s\ge 0}$ has the same law
  as $U(\cdot,s)_{s\geq 0}$.
  \item $U$ is symmetric: $U(\cdot, \cdot) \eqd - \; U(- \; \cdot, \cdot)$.
  \item For any $t \in [0, \infty)$, $\prob($There exists $x \in \Z$ such that $U(x, t) \neq \lim_{s \to t^-} U(x, t)) = 0$.
  \item $U(y, 0) = y)$ for all $y \in \Z$.
  \end{enumerate}
  \smallskip
  
  Moreover, for any sequence of times $\{t_n : n \in \nat\}$ such that $(n-1)/2 - t_n \to \infty$ as $n \to \infty$,
  $$
U^{k_n}_n (\cdot, t_n, \cdot) \cvgd U \qquad \mathas n \to \infty.
$$ 
\end{theorem}


We will need one more result from \cite{angel2017local} regarding the expected number of swaps at a given location in $U$.
Let $C(x, y)$ be the {\bf swap time} for particles $x$ and $y$ in the limit $U$. That is,
$$
C(x, y) = \sup \{ t: [U(x, t) - U(y, t)][U(x, 0) - U(y, 0)] > 0. \}.
$$
If $x$ and $y$ never cross in $U$, then $C(x, y) = \infty.$
On the event that $C(x, y) < \infty$, we can define the swap location 
$$
B(x, y) = \min \{U(x, C(x, y)), U(y, C(x, y))\}.
$$
For $i \in \Z$ and $t \in [0, \infty)$, we can now define 
$$
W(i, t) = \card{ \{(x, y) : B(x, y) = i, C(x, y) \le t  \}}.
$$
The function $W(i, t)$ counts the number of swaps at location $i$ up to time $t$.
\begin{theorem}
\label{T:swaps}
Let $i \in \Z$ and $t \in [0, \infty)$. Then $\expt W(i, t) = \frac{4t}{\pi}.$
\end{theorem}

\begin{section}{Existence of local speeds}
\label{S:existence}
In this section, we prove that particles have speeds in the local limit $U$. To do this, we first show that the environment of $U$ is stationary from the point of view of a particle.
\begin{theorem}
\label{T:time-stat}
For any particle $y$, and any time $t \in [0, \infty)$, we have that
\begin{equation}
\label{E:complicated-stat}
\lf[U(U(y, t) + \cdot, t, s) - U(y, t) \rg]_{s \ge 0} \eqd \lf[ U(y + \cdot, s) - U(y, 0) \rg]_{s \ge 0}. 
%
%
\end{equation}
This implies that all particle trajectories have stationary increments. That is, for any $y \in \Z$ and $t \in  [0, \infty)$, we have that
\begin{equation}
\label{E:stat-inc}
\lf[U(y, t + s) - U(y, t) \rg]_{s \ge 0} \eqd \lf[U(y, s) - U(y, 0) \rg]_{s \ge 0}.
\end{equation}
\end{theorem}

\begin{proof}
We will first prove \eqref{E:stat-inc} and then discuss what changes need to be made to prove the more general version \eqref{E:complicated-stat}. By spatial stationarity it suffices to prove \eqref{E:stat-inc} when $y = 0$. Let $A$ be any set in the Borel $\sig$-algebra generated by the Skorokhod topology on cadlag functions from $[0, \infty)$ to $\Z$. We compute
\begin{equation}
\label{E:U-split}
\prob \lf(\lf[U(0, t + s) - U(0, t) \rg]_{s \ge 0} \in A \rg)
\end{equation}
by splitting up the event above depending on the value of $U(0, t)$. This gives that \eqref{E:U-split} is equal to
\begin{align*}
\sum_{j \in \Z} \prob \lf(\lf[U(0, t + s) - j \rg]_{s \ge 0} \in A, \;\;U(0, t) = j \rg) &= \sum_{j \in \Z} \prob \lf(U( - j, t + s)_{s \ge 0} \in A, \;\; U(- j, t) = 0 \rg) \\
&= \prob \lf(U(U^{-1}(0, t), t + s)_{s \ge 0} \in A \rg) \\
&= \prob \lf(U(0, t, s)_{s \ge 0} \in A\rg) \\
&=  \prob (U(0, s)_{s \ge 0} \in A).
\end{align*}
The first equality above follows from spatial stationarity of $U$. The third equality is the definition of the time increment of $U$, and the final equality follows from the stationarity of time increments.

\medskip

The proof of \eqref{E:complicated-stat} is notationally more cumbersome, but follows the exact same steps in terms of splitting up the sum into the values of $U(y, t)$ and then applying spatial stationarity and stationarity of time increments.
\end{proof}

Now recall that $Q(x, t)$ is the number of swaps made by particle $x$ in $U$ in the interval $[0, t]$. Specifically,
$$
Q(x, t) = \card{\lf\{ r \in [0, t] : \lim_{s \to r^-} U(x, s) \ne U(x, r)\rg\}}.
$$
In order to apply the ergodic theorem to prove that particles have speeds, it is necessary to show that $Q(x, t) \in L^1$. To do this, we use a spatial stationarity argument to relate $Q(x, t)$ to $W(0, t)$, the number of swaps at location $0$ up to time $t$. Recall that $C(x, y)$ is the swap time of particles $x$ and $y$, and $B(x, y)$ is the swap location.

\begin{lemma}
\label{L:finite-exp-chunks}
In the local limit $U$, for any $x$ we have $\expt Q(x, t) = 8t/\pi$.
\end{lemma}

\begin{proof}
We have
\begin{align*}
\expt Q(x, t) &= \sum_{\substack{y \in \Z \\ y \ne x}} \sum_{i \in \Z} \prob(C(x, y) \le t, B(x, y) = i) \\
&= \sum_{\substack{y \in \Z \\ y \ne x}} \; \sum_{i \in \Z} \prob(C(x - i , y - i) \le t, B(x - i, y - i) = 0) \\
&= \sum_{\substack{y, z \in \Z \\ y \ne z}} \prob(C(z , y) \le t, B(z, y) = 0) \\
&= 2 \expt W(0, t).
\end{align*}
The second equality here comes from spatial stationarity of the process $U$. By Theorem \ref{T:swaps}, $\expt W(0, t) = 4t/\pi$, completing the proof.
\end{proof}
We can now prove every part of Theorem \ref{T:main-local} except for the fact that the speed distribution is bounded. First define
\begin{equation}
\label{E:St}
S(x, t) = \frac{U(x, t) - U(x, 0)}{t} 
\end{equation}
to be the average speed of particle $x$ up to time $t$.

\begin{corollary}
\label{C:speed-exist}
For every $x \in \Z$, the limit
$$
S(x) = \lim_{t \to \infty} S(x, t)  \qquad \text{exists } \as.
$$
$S(x)$ is a symmetric random variable with distribution $\mu$ independent of $x$. 
Moreover, the random function $S: \Z \to \real$ is stationary and mixing of all orders with respect to the spatial shift $\tau$ given by  $\tau S(x) = S(x + 1)$.
\end{corollary}

\begin{proof}
The function $U(x, 1) - U(x, 0)$ is in $L^1$ by Lemma \ref{L:finite-exp-chunks} since $|U(x, 1) - U(x, 0)|$ is bounded by $Q(x, t)$. The existence of the limit follows by the stationary of particle increments in Theorem \ref{T:time-stat} and Birkhoff's ergodic theorem.

\medskip

 The fact that the distribution of $S(x)$ is independent of $x$ follows from spatial stationarity of $U$, and all the properties of $S(\cdot)$ come from the corresponding properties of $U$.
\end{proof}

\end{section}
\section{Boundedness of local speeds}
\label{S:bounded-speed}

In this section, we prove that the local speed distribution $\mu$ is bounded, completing the proof of Theorem \ref{T:main-local}.
\begin{theorem}
\label{T:bounded-speed}
$\supp(\mu) \sset [-\pi, \pi]$.
\end{theorem}

We first prove a lemma concerning the existence of fast particles at finite times in the local limit $U$.

\begin{lemma}
\label{L:box-implies-diagonals}
For every $\ep > 0$, we have that
$$
\liminf_{t\to \infty}\prob (\text{There exists } x < 0 \mathsuchthat \; U(x, t) > (\pi + \ep)t ) < 1.
$$
%
\end{lemma}

\begin{proof} Let $A_{t, \ep}$ be the event in the statement of the lemma.
Suppose that for some $\ep > 0$, that 
$
\lim_{t\to \infty}\prob (A_{t, \ep}) = 1.
$
Fix $\de > 0$, and let $h \in \nat$ be large enough so that
\begin{equation}
\label{E:h-lower-bd}
 \frac{h \ep}{2(\pi + \ep)\sqrt{1 - \de^2}} \ge 2.
\end{equation}
For each $\al \in (-1, 1)$, define
$$
t_{\al, n} = \floor{{n \choose 2}\frac{\arcsin(\al) + \pi/2}{\pi + \ep/2}}\;, \qquad t_{\al,n}^+=t_{\al,n}+  \frac{hn}{(\pi + \ep)\sqrt{1 - \al^2}} ,\qquad j_{\al, n} = \floor{\frac{n(1+\al)}2}.
$$
For each $n \in \nat$ and $\al \in (-1, 1)$, consider the random variable 
$$
Z_{\al, n} := \indic \bigg( \exists x \in [1, n] \text{ such that }
\sig^n \lf(x, t_{\al, n} \rg) < j_{\al, n}, \;
\sig^n \lf(x, t_{\al, n}^+ \rg  ) >j_{\al, n}+ h \bigg).
$$
When $Z_{\al, n}=1$, there exists an increasing subsequence of swaps in the time interval $[t_{\al,n},t_{\al,n}^+]$ at locations $j_{\alpha,n},j_{\alpha,n}+1, \ldots, j_{\alpha,n}+h-1$. Consider the set
$$
A_n=\{\alpha\in(-1+\delta,1-\delta)\,:\,j_{\alpha,n}\in h \mathbb Z\}.
$$
A straightforward computation shows that for all large enough $n$, when $\alpha,\alpha'\in A_n$ and $j_{\alpha,n}\not=j_{\alpha',n}$ then the time intervals
$[t_{\alpha,n}, t_{\alpha,n}^+]$ and $[t_{\alpha',n}, t_{\alpha',n}^+]$ are disjoint (this is where condition \eqref{E:h-lower-bd} is used).  This implies that if $\al_1 < \al_2 \ddd < \al_m$ is a sequence in $A_n$ with $j_{\al_i} \neq j_{\al_{i+1}}$ for all $i$, and $Z_{\al_i, n} = 1$ for all $i$, then the increasing subsequences for each $\al_i$ can be concatenated to get an increasing subsequence of length $mh$ in the time interval $[t_{\al_1, n} , t_{\al_m, n}^+]$.

Now we can also assume that $n$ is large enough so that
$$
t_{\al, n}^+ \le {n \choose 2}\frac{\pi}{\pi + \ep/2}
$$
whenever $\al \in A_n$. Since the intervals $\{\alpha:j_{\alpha,n}=j\}$ are of
Lebesgue measure $2/n$, the longest increasing subsequence in the first $\pi/(\pi + \ep/2)$ fraction of swaps satisfies 
\begin{align}
\label{E:L-not-exp-yet} 
L _n \lf( \frac{\pi}{\pi + \ep/2} \rg) &\ge \frac{nh}2\int_{A_n} Z_{\al, n} d\al.
\end{align}
Observe that the boundary of $A_{t, \ep}$ in the space of swap functions is contained in the set of swap functions that have a swap at time $t$. This is a null set in $\prob_U$ by Theorem \ref{T:local} (iv), so $A_{t, \ep}$ is a set of continuity for $U$. The weak convergence in Theorem \ref{T:local} then implies that $\expt Z_{\al, n} \to \prob(A_{h/(\pi + \ep), \ep})$ for every $\al$.

\medskip

Choose $h$ large enough so that $\prob(A_{h/(\pi + \ep), \ep}) > 1 - \de$. Then $\lim_{n \to \infty} \expt Z_{\al, n} \ge 1-\de$ for every $\al$, and so by bounded convergence, 
\begin{align*}
2-2\delta = \lim_{n \to \infty} \int_{-1+\delta}^{1-\delta}\frac{\expt Z_{\alpha,n}}{1-\delta}  {\wedge} 1\,d\alpha \le \liminf_{n\to\infty} \int_{A_n} \frac{\expt Z_{\alpha,n}}{1-\delta}\,d\alpha+ (2-2\delta) \lf(1-\frac{1}h \rg).
\end{align*}
The last term is the limiting Lebesgue measure of $(-1+\delta,1-\delta)\setminus A_n$. Taking expectations in \eqref{E:L-not-exp-yet} and applying Fubini's Theorem then gives that for large enough $n$,
\begin{align*}
\expt L _n \lf( \frac{\pi}{\pi + \ep/2} \rg) \ge \frac{n}{2}(2-2\delta)(1-\delta).
\end{align*}
Taking $\de$ small enough given $\ep$ then contradicts Theorem \ref{T:subsequence}.
\end{proof}

We now show that the condition in Lemma \ref{L:box-implies-diagonals} implies that the speed is bounded, completing the proof of Theorem \ref{T:main-local}.

%
%
%

\begin{proof}[Proof of Theorem \ref{T:main-local}.]
Suppose that $\mu(\pi, \infty) > 0$, and fix $\ep > 0$ such that $\mu(\pi + 3\ep, \infty) > 0$. Then for any fixed $\de > 0$, by spatial ergodicity we can find an $m \in \nat$ such that
$$
\prob \lf( \text{There exists } x \in [-m, -1] \mathsuchthat \; S(x) > \pi + 3\ep \rg) > 1 - \de/2.
$$
Then there exist some $t_0 > 0$ such that for every $t > t_0$,
$$
\prob \lf( \text{There exists } x \in [-m, -1] \mathsuchthat \; \frac{U(x, t) - x}t > \pi + 2\ep \rg) > 1 - \de.
$$
If $t$ is chosen large enough so that $(\pi + 2\ep)t - m > (\pi + \ep)t$, then the above inequality immediately implies that
$
\prob \lf(A_{t, \ep} \rg) > 1 - \de.
$
As $\de$ was chosen arbitrarily, this contradicts Lemma \ref{L:box-implies-diagonals}.
\end{proof}

\section{Local swap rates}
\label{S:local-swap-rates}
The main goal of this section is to prove Theorem \ref{T:swap-rate}. We first recall the statement here.

\begin{customthm}{1.8}
Let $x \in \Z$, and let $Q(x, t)$  be the number of swaps made by particle $x$ up to time $t$ in the local limit $U$. Let $S(x)$ be the asymptotic speed of $x$, and let $\mu$ be the local speed distribution, as in Theorem \ref{T:main-local}. Then
$$
\lim_{t \to \infty} \frac{Q(x, t)}t = \int |y - S(x)|d\mu(y) \qquad \text{almost surely and in $L^1$.}
$$

\end{customthm}

This theorem allows us to control the number of swaps in $\sig^n$ between ``typical particles" moving with local speed at most $\pi + \ep$. This will imply a lower bound on the number of swaps in a random sorting network made by particles with speed greater than $\pi + \ep$, which in turn will allow us prove that limiting trajectories are $\pi\sqrt{1-y^2}$-Lipschitz. Specifically, we will need the following corollary in our proof of Theorem \ref{T:main}:

\begin{corollary} 
\label{C:av-swap-rate}
(i) For any $x$, the following statement holds almost surely for the local limit $U$.
$$
\lim_{t \to \infty} \frac{Q(x, t)}t \in [0, \pi].
$$

(ii) Let $X$ and $X'$ be two independent samples from the local speed distribution $\mu$. Then
$$
\expt |X - X'| = \expt\lf[\lim_{t \to \infty} \frac{Q(0, t)}{t} \rg]= \frac{8}{\pi}.
$$
\end{corollary}

The intuition behind Theorem \ref{T:swap-rate} is very simple. Since particles in $U$ have asymptotic speeds and $U$ is spatially ergodic, we can imagine $U$ as a collection of particles moving along linear trajectories with independent slopes sampled from the local speed distribution. With this heuristic, the quantity $\frac{Q(x, t)}t$ can be estimated as a sum of two integrals for large $t$: 
$$
\int_{S(x)}^\pi \mu(y, \infty) dy + \int_{-\pi}^{S(x)} \mu(-\infty, y)dy.
$$

\medskip

To make this intuition rigorous, we first prove a corresponding theorem for lines, and then use these lines to approximate the trajectory of $x$ in $U$. Let $L$ be a line with slope $c \in \real$ given by the formula $L(t) = ct + d$. Define 

\begin{equation*}
C^+(x, L, t) = \begin{cases} 
1, \qquad \text{if} \;\; U(x, 0) \le L(0) \mathand U(x, t) > L(t).\\
0, \qquad \text{else}.
\end{cases}
\end{equation*}

Define $C^+(L, t)$, the number of net upcrossings of the line $L$ in the interval $[0, t]$, by $C^+(L, t) = \sum_{x \in \Z} C^+(x, L, t)$. We then have the following proposition:

\begin{prop}
\label{P:line-rate} Let $L(t) = ct + d$. Then
\begin{equation*}
\frac{C^+(L, t)}t \to \int (y - c)^+d\mu(y) \qquad \as \text{ and in $L^1$}.
\end{equation*}
\end{prop}

We first show that the limit always exists.

\begin{lemma}
\label{P:line-rate-exists}
For any line $L(t) = ct + d$, there exists a random $C^+(L) \in L^1(\real)$ such that 
\begin{equation*}
\frac{C^+(L, t)}t \to C^+(L) \qquad \as \text{ and in $L^1$}.
\end{equation*}
\end{lemma}

To prove this lemma, we introduce a space-time shift $\tau_{a, t}$ on the space of swap functions. Here $a \in \Z$ and $t \in [0, \infty)$. The shift $\tau_{a, t}$ shifts the swap function $U$ by $a$ in space and then looks at the increment starting from time $t$:
$$
\tau_{a, t}U(x, s) = U(x + a, t, s) - a.
$$
\begin{proof}
This will follow immediately from Kingman's subadditive ergodic theorem. We first consider the case $c   \ne 0$. Define $\tau := \tau_{\text{sgn}(c), |c|^{-1}}$. Since $U$ is stationary in both space and time, $\tau U \eqd U$.

\medskip

Now let $f_n(U)= C^+(L, |c|^{-1}n)$. The sequence $f_n$ satisfies a subadditivity relation with respect to the shift $\tau$ given by
$$
f_{n+m}(U) \le f_n(U) + f_m(\tau^nU).
$$
Moreover, $f_n(U) \in L^1$ for all $n$. To see this, observe that if $x \le L(0)$ and $U(x, t) > L(t)$, then either $L(t) < x \le L(0)$, or else $x$ swaps at a time $s \in [0, t]$ at some position in the spatial interval $[L(t) -1, L(0)]$ (if $c < 0$) or $[L(0) -1, L(t)]$ (if $c > 0)$. 

\medskip

The expected number of particles that can make swaps in this region is finite by Theorem \ref{T:swaps}, and the number of particles $x$ with $L(t) < x \le L(0)$ is bounded by $|c|t + 1$. 

\medskip

Therefore by Kingman's subadditive ergodic theorem, the sequence $f_n(U)/n$ has an almost sure and $L^1$ limit $C^+(L) \in L^1(\real)$, and therefore so does $\frac{C^+(L, t)}t$. To modify this in the case when $c=0$, consider the usual time-shift by 1. 
\end{proof}

To find the value of the limit in Lemma \ref{P:line-rate-exists} we introduce a collection of approximations of $C^+(L, t)$. Let $A(x, \ep, s)$ be the event where
$$
\card{\frac{U(x, t') - U(x, 0)}t - S(x)} < \ep \mathforall t' > s.
$$
 For any $s \in [0, \infty)$ and $\ep > 0$, define 
 $$
 C^+_{s, \ep}(L, t) = \sum_{x \in \Z} C^+(x, L, t) \indic(A(x, \ep, s)).
$$
\begin{lemma}
\label{L:C-ep-C-1}
For any $\ep > 0$, we have that
\begin{equation}
\label{E:lim-want}
\lim_{t \to \infty} \frac{C^+(L, t)}t = \lim_{s \to \infty}  \limsup_{t \to \infty} \frac{C^+_{s, \ep}(L, t)}t \qquad \as.
\end{equation}
\end{lemma}

\begin{proof} We show that for any $\ep > 0$,
\begin{equation}
\label{E:C-ep-C}
\lim_{s \to \infty}  \liminf_{t \to \infty} \frac{C^+(L, t) - C^+_{s, \ep}(L, t)}t = 0\qquad \as.
\end{equation}
We have 
\begin{align}
\nonumber
0 &\le \frac{C^+(L, t)}t -\frac{C^+_{s, \ep}(L, t)}t \\
\nonumber
&\le \frac{1}t\sum_{x = d- (2\pi-c)t}^{d} C^+(x, L, t) \indic(A(x, \ep, s)^c) + \frac{1}t\sum_{x < d - (2\pi-c)t} C^+(x, L, t) \\
\label{E:two-terms-1} &\le \frac{1}{t}\sum_{x = d- (2\pi -c)t}^{d}\indic(A(x, \ep, s)^c) +\frac{1}t\sum_{x < d - (2\pi-c)t} C^+(x, L, t).
\end{align}
Define 
$$
B(L, t) =  \sum_{x < d - (2\pi-c)t} C^+(x, L, t), \qquad \text{ and let } \qquad B(L) = \liminf_{t \to \infty} \frac{B(L, t)}t.
$$
Birkhoff's ergodic theorem implies that as $t$ approaches $\infty$, the first term in \eqref{E:two-terms-1} approaches $
|2\pi - c|\prob(A(x, \ep, s)^c).
$
Therefore
\begin{align*}
0 \le \liminf_{t \to \infty} \frac{C^+(L, t) - C^+_{s, \ep}(L, t)}t \le |2\pi - c|\prob(A(x, \ep, s)^c) + B(L).
\end{align*}
We have that
$
\prob(A(x, \ep, s)^c) \to 0
$
as $s \to \infty$. Therefore to prove \eqref{E:C-ep-C}, it is enough to show that $B(L) = 0$ almost surely. We first show that it is almost surely constant. 

\medskip

Letting $[L + i](t) = ct + d + i$, we claim that $|B(L, t) - B(L + i, t)| \le 2i.$ 
To see this when $i > 0$, first observe that the only particles that can upcross $L + i$ but not $L$ in the interval $[0, t]$ are those that start between $L(0)$ and $[L +i](0)$. There are at most $i$ such particles. Similarly, the only particles that can upcross $L$ but not $L+i$ in the interval $[0, t]$ are those that are between $L(t)$ and $[L+i](t)$ at time $t$. Again, there are at most $i$ such particles. This proves the desired bound. Similar reasoning works when $i < 0$.

\medskip

Therefore the limit $B(L + i)$ is the same for all $i$, and so the random variable $B(L)$ lies in the invariant $\sig$-algebra of the spatial shift. By spatial ergodicity, $B(L)$ is almost surely constant. We have that
\begin{align}
\nonumber \prob \lf( \frac{B(L, t)}t > 0 \rg) &= \prob \big(\text{There exists $x < d -(2\pi -c)t$  such that $U(x, t) > d + ct$} \big) \\
\label{E:prob-B-+}
&=\prob \lf(\text{There exists $x < 0$  such that $U(x, t) > 2\pi t$} \rg),
\end{align}
where the second equality follows by spatial stationarity. By Lemma \ref{L:box-implies-diagonals}, \eqref{E:prob-B-+} does not approach $1$ as $t \to \infty$, and thus $B(L) = 0$ almost surely. This proves \eqref{E:C-ep-C}.

\medskip

The almost sure existence of the limit $\lim_{t \to \infty} C^+(L, t)/t$ by Lemma \ref{P:line-rate-exists} allows us to rearrange \eqref{E:C-ep-C} to get \eqref{E:lim-want}. 
\end{proof}

We now establish bounds on the limits of $C^+_{s, \ep}$. For this we need the following lemma about sequences. The proof is straightforward, so we omit it.

\begin{lemma}
\label{L:sequence-avg}
Let $(a_n : n \in \nat)$ be a sequence such that
$$
\lim_{n \to \infty} \frac{1}{n + 1} \sum_{i=0}^n a_i = a.
$$
Then for any sequence $j(m) \in \Z_+$ such that $j(m)/m \to k > 0$, and any $c > 0$, we have that
$$
\lim_{m \to \infty} \frac{1}{m + 1} \sum_{i = j(m)}^{j(m) + cm} a_i = ca.
$$
\end{lemma}

\begin{lemma}
\label{L:C-ep-C}
For any line $L(t) = ct + d$ and any $\ep > 0$, we have that
\begin{align}
\label{E:limsup}
\int_{c + \ep}^\infty \mu(y, \infty) dy \le
\lim_{s \to \infty} \limsup_{t \to \infty} \frac{C^+_{s, \ep}(L, t)}t& \le \int_{c - \ep}^\infty \mu(y, \infty)dy.\end{align}
%
\end{lemma}

\begin{proof}
We will prove this when $d = 0$. The result follows for all other $d$ by the spatial stationarity of $U$. 
We can write $C^+_{s, \ep}(L, t)$ as follows. 
\begin{align*}
C^+_{s, \ep}(L, t) &= \sum_{ x < 0}  \indic(U(x, t) - x > ct - x)\indic(A(x, \ep, s)) \\
&=    \sum_{ x < 0}  \indic\left(\frac{U(x, t) - x}t  - c > -\frac{x}t \right)\indic(A(x, \ep, s)). 
\end{align*}
On the event $A(x, \ep, s)$, for $t > s$, $\frac{U(x, t) - x}t \in (S(x) - \ep, S(x) + \ep)$. This gives the following two almost sure bounds on $C^+_{s, \ep}(L, t).$
\begin{align*}
C^+_{s, \ep}(L, t) &\le \sum_{x < 0}  \indic\left(S(x) - (c - \ep) > -\frac{x}t \right)\indic(A(x, \ep, s)) \\
&\le \sum_{x \in [-(\pi - c + \ep)t, 0)}  \indic\left(S(x) - (c - \ep) > -\frac{x}t \right)\qquad \mathand \\
C^+_{s, \ep}(L, t) &\ge \sum_{x \in [-(\pi - c - \ep)t, 0)}  \indic\left(S(x) - (c + \ep) > -\frac{x}t \right)\indic(A(x, \ep, s)).
\end{align*}
Here the change in the range of $x$-values follows since $S(x) \in [-\pi, \pi]$ almost surely for every $x$ by Theorem \ref{T:bounded-speed}. 
We now prove the upper bound in \eqref{E:limsup}.
For any $m \in \Z$, we have the following:
\begin{align*}
&\sum_{x \in [-(\pi - c + \ep)t, 0)}  \indic\left(S(x) - (c - \ep) > -\frac{x}t \right) \\
&\qquad \qquad \le \sum_{z=0}^{\lceil (\pi -c + \ep)m \rceil} \sum_{x = 0}^{\ceil{t/m} - 1} \indic\left(S\big(-x - z\ceil{t/m}\big) > \frac{z}m + c - \ep \right).
\end{align*}
Applying Birkhoff's ergodic theorem and Lemma \ref{L:sequence-avg} implies that for each $z$, almost surely as $t \to \infty$, we have
\begin{align*}
\frac{1}{t} \sum_{x = 0}^{\ceil{t/m} -1} \indic\left(S\big(-x - z\ceil{t/m}\big) > \frac{z}m + c - \ep\right) &\to \frac{1}m\mu \lf(\frac{z}{m} + c - \ep, \infty \rg).
\end{align*}
Summing over $z$, we get that 
$$
\limsup_{t \to \infty} \frac{C^+_{s, \ep}(L, t)}t \le \frac{1}{m} \sum_{z=0}^{\ceil{(\pi + \ep - c)m}}\mu \lf(\frac{z}{m} + c - \ep, \infty \rg) \qquad \as \text{ for every $m$}.
$$
Taking $m \to \infty$, the above Riemann sum converges to the corresponding integral, proving the upper bound in \eqref{E:limsup}. To prove the lower bound in \eqref{E:limsup}, first observe that for any $m \in \Z_+$, we have
\begin{align*}
\sum_{x \in [-(\pi - c - \ep)t, 0)}  &\indic\left(S(x) - (c + \ep) > -\frac{x}t \rg)\indic(A(x, \ep, s)) \\
&\ge \sum_{z=0}^{ \floor{(\pi - c - \ep)m} - 1} \sum_{x = 0}^{\floor{t/m} - 1} \indic\bigg(S(x + z\floor{t/m}) - (c + \ep) > \frac{z+1}{m}\mathand \\
&\qquad \qquad |S(x + z\floor{t/m}, t') - S(x + z\floor{t/m})| < \ep \mathforall t' > s\bigg).
\end{align*}
In the above inequality, we have used the notation $S(x, t)$ for the average speed of particle $x$ up to time $t$ (see the definition in Equation \eqref{E:St}).
From here we take $t \to \infty$, and apply Lemma \ref{L:sequence-avg} as in the proof of the upper bound in \eqref{E:limsup}. This gives the almost sure bound
\begin{equation}
\label{E:almost-reimann-ready}
\limsup_{t \to \infty} \frac{C^+_{s, \ep}(L, t)}t \ge \frac{1}{m} \sum_{z=0}^{\floor{(\pi - \ep - c)m} -1}\prob \lf(S(0) - (c + \ep) \ge \frac{z + 1}{m} \mathand |S(0, t') - S(0)| < \ep \mathforall t' > s\rg).
\end{equation}
Now observe that as $s \to \infty$, 
$$ 
\prob \lf(S(0) - (c + \ep) \ge \frac{z + 1}{m} \mathand |S(0, t') - S(0)| < \ep \mathforall t' > s\rg) \to \mu \lf(\frac{z + 1}{m} + c + \ep , \infty\rg).
$$
Therefore taking $s \to \infty$ in \eqref{E:almost-reimann-ready}, and then letting $m$ tend to infinity  proves the lower bound in \eqref{E:limsup}.
\end{proof}

\begin{proof}[Proof of Proposition \ref{P:line-rate}.]
Applying Lemmas \ref{L:C-ep-C-1} and \ref{L:C-ep-C} gives that for any $\ep > 0$, that
$$
\int_{c + \ep}^\infty \mu(y, \infty)dy \le \lim_{t \to \infty} \frac{C^+(L, t)}t \le \int_{c + \ep}^\infty \mu(y, \infty)dy
$$
almost surely.
Taking $\ep$ to $0$ then completes the proof of almost sure convergence.
The fact that convergence also takes place in $L^1$ follows from Lemma \ref{P:line-rate-exists}.
\end{proof}

We can analogously define $C^-(L, t)$ as the number of net downcrossings of the line $L$ by the time $t$, and define $C(L, t) = C^+(L, t) + C^-(L, t)$. By the symmetry of the local limit $U$, analogues of Proposition \ref{P:line-rate} hold for $C^-(L, t)$ and $C(L, t)$.


\begin{theorem}
\label{T:line-rate}
Let $L(t) = ct + d$. Then as $t \to \infty$, we have that
$$
\frac{C^+(L, t)}{t} \to \int(y-c)^+d\mu(y), \;\;\;\;  \frac{C^-(L, t)}{t} \to  \int(y-c)^-d\mu(y), \; \;\;\; \frac{C(L, t)}{t} \to \int|y-c|d\mu(y).
$$
All three convergences are both almost sure and in $L^1$.
\end{theorem}

\begin{proof}[Proof of Theorem \ref{T:swap-rate}.]
By Theorem \ref{T:time-stat}, the process of swap times for the particle $x$ is stationary in time. Moreover, $Q(x, 1) \in L^1$ by Lemma \ref{L:finite-exp-chunks}. Therefore we can apply Birkhoff's ergodic theorem to get that $Q(x, t)/t$ converges both almost surely and in $L^1$ to a (possibly random) limit. We now identify that limit.

\medskip

Let $L_q(t) = qt + x$. By Theorem \ref{T:line-rate}, we have that with probability $1$, 
\begin{equation}
\label{E:C-L-q}
\frac{C(L_q, t)}{t} \to \int |y - q|d\mu(y) \qquad \text{for every $q \in \rat$}.
\end{equation}
 At time $t$, there are fewer than $|U(x, t) - L_q(t)| + 1$ particles that either have crossed the line $L_q(t)$ by time $t$ but have not swapped with particle $x$, or have swapped with particle $x$ by time $t$ but have not crossed the line $L_q(t)$. Therefore we have that almost surely,
\begin{equation}
\label{E:t-finite}
\card{\frac{Q(x,t) - C(L_q, t)}t} \le \frac{|U(x, t) - L_q(t)| + 1}t = |S(x) - q| + o(1).
\end{equation}
The last equality follows from Theorem \ref{T:main-local}. Letting $t \to \infty$ in \eqref{E:t-finite}, the convergence in \eqref{E:C-L-q} implies that almost surely,
\begin{equation*}
\lf| \lim_{t \to \infty} \frac{Q(x, t)}t - \int |y - q|d\mu(y) \rg| \le |S(x) - q| \qquad \text{for every $q \in \rat$}. 
\end{equation*}
By the continuity of the function $F(z) = \int |y - z|d\mu(y)$, this implies that 
\[
\lim_{t \to \infty} \frac{Q(x, t)}t = \int |y - S(x)|d\mu(y) \qquad \as. \qquad \qedhere
\]

\end{proof}

\begin{proof}[Proof of Corollary \ref{C:av-swap-rate}.]
 For (i), observe that for $c \in [-\pi, \pi]$,
 we have that
 \begin{align*}
 2\int |y - c| d\mu(y) &=  \int \big( |y - c| + |y + c| \big) d\mu(y) \\
 &\le \int\big( |y - \pi| + |y + \pi| \big) d\mu(y) = \int 2\pi d\mu(y) = 2\pi.
 \end{align*}
 Here the first equality follows from the symmetry of $\mu$, and the second equality follows since $\supp(\mu) \sset [-\pi, \pi]$.
Therefore by Theorem \ref{T:bounded-speed} and Theorem \ref{T:swap-rate},
$$
0 \le \lim_{t \to \infty} \frac{Q(x, t)}t \le \pi \qquad \as.
$$
For (ii), Birkhoff's ergodic theorem implies that
\begin{equation*}
\expt\lf[\lim_{t \to \infty} \frac{Q(x, t)}{t}\rg] = \expt Q(x, 1).
\end{equation*}
Here the left hand side above is equal to $\expt|X - X'|$ by Theorem \ref{T:swap-rate}, where $X$ and $X'$ are independent random variables with distribution $\mu$. By Lemma \ref{L:finite-exp-chunks}, the right hand side is equal to $8/\pi$.
\end{proof}

\section{Limiting trajectories are Lipschitz}
\label{S:Lipschitz}

Recall that a path $y(t)$ is $\pi\sqrt{1 - y^2}$-Lipschitz if it is absolutely continuous, and if $|y'(t)| \le \sqrt{1 - y^2}$ for almost every time $t$. The goal of this section is to prove Theorem \ref{T:main}, showing that weak limits of the trajectory random variables $Y_n$ are supported on $\pi\sqrt{1 - y^2}$-Lipschitz paths. 

\medskip

Theorem \ref{T:bounded-speed} allows us to conclude that most particles move with bounded local speed most of the time. In order to translate this into a global speed bound we need to bound the amount of particle movement during the times that particles are not moving with bounded speed.
For $p, q \in [0, 1]$, and a path $y: [0, 1] \to [-1, 1]$, define 
$$
m_{p ,q}(y) = \inf \{ |y(t)| : t \in [p, q]\}.
$$

\begin{lemma}
\label{L:bad-box}
For any $\ep > 0$ and $q \in (0, 1]$, the following holds:
$$
\frac{1}n \expt \sum_{x=1}^n \lf(|\sig^n_G(x, q) - \sig^n_G(x, 0)| - (\pi + \ep)q\sqrt{1 - m^2_{0, q}(\sig^n_G(x, \cdot))} \rg)^+ \to 0 \qquad \mathas n \to \infty.
$$
\end{lemma}

\begin{proof}
We first reduce the lemma to a statement about the number of swaps made by fast-moving particles. Fix $t \in [0, \infty)$. For $j \in \{0, 1, \ddd, \floor{(n-1)q/2t} \}$, define 
$$
\De_n =  \frac{2(t + 1)}{n-1}, \qquad  t_{n, j} = \frac{\floor{ntj}}{{n \choose 2}} \qquad \mathand \qquad t^+_{n, j} = t_{n, j} + \De_n.
$$
This intervals $[t_{n, j}, t^+_{n, j}]$ cover the interval $[0, q]$ with some overlap. The reason for adding in the overlap is so that all intervals are the same length and start at multiples of ${n \choose 2}^{-1}$. This is necessary for applying time stationarity of sorting networks. Let $Q_n(x, t_1, t_2)$ be the number of swaps made by particle $x$ in the interval $[t_1, t_2]$ in the random sorting network $\sig^n$. Then for large enough $n$, we have that
\begin{align}
\nonumber
&\frac{1}n \expt  \sum_{x=1}^n \lf(|\sig^n_G(x, q) - \sig^n_G(x, 0)| - (\pi + \ep)q\sqrt{1 - m^2_{0, q}(\sig^n_G(x, \cdot))} \rg)^+ \\
\nonumber
&\qquad \le \frac{1}n \expt \sum_{x=1}^n \sum_{j=0}^{\floor{(n-1)q/2t}} \lf(|\sig^n_G(x, t_{n, j + 1}) - \sig^n_G(x, t_{n, j})| - \frac{2(\pi + \ep/2)t}{n-1}\sqrt{1 - m^2_{0, q}(\sig^n_G(x, \cdot))} \rg)^+.
\end{align}
This inequality comes from using the convexity of the function $f(x) = x^+$ and the triangle inequality. We can now bound the distance $|\sig^n_G(x, t_{n, j + 1}) - \sig^n_G(x, t_{n, j})|$ by the number of swaps $Q_n(x, t_{n, j}, t_{n, j+1})$ made by particle $x$ in that interval. Then using that $Q_n(x, t, \cdot)$ is an increasing function and that $t_{n, j+1} \le t_{n, j}^+$, the right hand side above can be bounded by
\begin{align} 
\label{E:ready-to-Q-bd}
\frac{1}n \expt \sum_{x=1}^n \sum_{j=0}^{\floor{(n-1)q/2t}} \lf(\frac{2Q_n(x, t_{n, j}, t_{n, j}^+)}n - \frac{2(\pi + \ep/2)t}{n-1}\sqrt{1 - \lf[\sig^n_G(x, t_{n, j})\rg]^2} \rg)^+
\end{align}

\medskip

It is enough to show that for any $\de > 0$, there exists a $t \in [0, \infty)$ such that for large enough $n$, \eqref{E:ready-to-Q-bd} is bounded by $\de$. By time stationarity of sorting networks, it is enough to show that there exists some $t$ such that for all large enough $n$, the quantity
\begin{align*}
\scrF_n := \expt \sum_{x=1}^n  \lf(\frac{2Q_n(x, 0, \De_n)}n - \frac{2(\pi + \ep/2)t}{n-1}\sqrt{1 - \lf(\frac{2x}n - 1\rg)^2} \rg)^+
\end{align*}
is at most $t\de$. 
For $\al \in (-1, 1)$, let $j_{n, \al} = \floor{\frac{n(\al + 1)}2}$, and define the random variable 
$$
Z^n_{\al, t} = Q_n\lf(j_{n, \al}, 0, \De_n\rg)\indic \lf( Q_n\lf(j_{n, \al}, 0, \De_n\rg) < (\pi + \ep/2)t\sqrt{1 - \lf(\frac{2j_{n, \al}}n - 1\rg)^2} \rg).
$$
We can bound $\scrF_n$ in terms of the random variables $Z_{\al, t}^n$:
\begin{align}
\nonumber
\nonumber
\scrF_n &\le \frac{2}n \expt \sum_{x=1}^n  Q_n\lf(x, 0, \De_n\rg) \indic\lf( Q_n\lf(x, 0, \De_n\rg) \ge (\pi + \ep/2)t\sqrt{1 - \lf(\frac{2x}n - 1\rg)^2}\rg) \\
\label{E:Z-want}
&= 4t - \expt \int_{-1}^1 Z^n_{\al, t} d\al.
\end{align}
It remains to bound $\expt \int_{-1}^1 Z^n_{\al, t} d\al.$ Recall that in the local limit, that $Q(0, t)$ is the number of swaps made by particle $0$ in the interval $[0, t]$. Define the random variable
$$
Z(t) := Q(0, t + 1) \indic(Q(0, t + 1) < (\pi + \ep/2)t).
$$
We can think of $Z$ as a function on the product space $\scrA \X [0, \infty)$, where $\scrA$ is the space of swap functions. Thought of in this way, if $U_n \to U$ in $\scrA$, and $t_n \to t$, then $Z (U_n, t_n) \to Z(U, t)$ as long as particle $0$ does not swap in $U$ at time $t + 1$. 

\medskip

For any $t$, the probability that the local limit $U$ has a swap at time $t$ is $0$ by Theorem \ref{T:local} (iv). Therefore by the weak convergence in Theorem \ref{T:main-local}, since $2j_{n, \al}/n - 1 \to \al$ as $n \to \infty$ for any $\al \in (-1, 1)$, we get that 
$$
Z^n_{\al, t} \cvgd Z(\sqrt{1- \al^2}t).
$$
Therefore by the bounded convergence theorem, we have that
\begin{equation}
\label{E:Z-to-A}
 \lim_{n \to \infty} \int_{-1}^1 \expt Z^n_{\al, t} d\al = t  \int_{-1}^1 \expt Z(\sqrt{1 - \al^2}t) d\al.
\end{equation}
Now by Corollary \ref{C:av-swap-rate}, $Z(t)/t \to \int |S(0) - y|d \mu(y)$, and so by the bounded convergence theorem and Corollary \ref{C:av-swap-rate} again, $\expt Z(t)/t \to 8/\pi.$ Therefore we have that
$$
\int_{-1}^1 \frac{\expt Z(\sqrt{1 - \al^2}t)}t d\al \xrightarrow[\;\; t \to \infty \;\;]{} \frac{8}\pi \int_{-1}^1 \sqrt{1 - \al^2} d\al = 4,
$$
where the bounded convergence theorem is once again used to establish the limit. Combining the above convergence with \eqref{E:Z-to-A} implies that there exists a $t$ such that for all large enough $n$,
$$
\int_{-1}^1 \expt Z^n_{\al, t} d\al \ge (4 - \de)t.
$$
Using Fubini's Theorem and then plugging the above inequality into \eqref{E:Z-want} then gives that $\scrF_n \le t\de$ for large enough $n$, as desired.
\end{proof}

\begin{proof}[Proof of Theorem \ref{T:main}.]
By Lemma \ref{L:bad-box} and Markov's inequality, for any $\ep > 0$ $\mathand q \in [0, 1]$, we have that
\begin{equation}
\label{E:translate}
\lim_{n \to \infty} \prob \lf(|Y_n(q) - Y_n(0)| \le \pi q\sqrt{1 - m^2_{0, q}(Y_n)} + \ep \rg) = 1.
\end{equation}
Moreover, by time-stationarity of sorting networks, the above holds with any $p< q$ inserted in place of $0$. Now for any $p, q \in [0, 1]$ and $C \in \real$, the set $\{|f(p) - f(q)| \le C\}$ is closed in $\scrD$ under the uniform norm. Therefore since any subsequential limit $Y$ of $Y_n$ is supported on continuous paths by Lemma \ref{L:precompact},  by \eqref{E:translate},
$$
\prob \lf(\text{For all    } p, q \in \rat \cap [0, 1] \mathand k \in \nat, \text{ we have } \frac{|Y(q) - Y(p)|}{q - p} \le \sqrt{1 - m^2_{p, q}(Y)} + \frac{1}k \rg) = 1.
$$
Since $Y$ is almost surely continuous, this implies that
$$
\prob \lf(\text{For all    } s, t \in [0, 1], \text{ we have } \frac{|Y(t) - Y(s)|}{t - s} \le \pi\sqrt{1 - m^2_{t, s}(y)} \rg) = 1.
$$
This condition is equivalent to $Y$ being almost surely $\sqrt{1 - y^2}$-Lipschitz.
\end{proof}

\section{Elliptical support and sine curve trajectories at the edge}
\label{S:corollaries}

In this section, we use Theorem \ref{T:main} to prove Theorems \ref{T:main-3} and \ref{T:main-2}. Recall that $\eta^n_t$ is the permutation matrix measure at time $t$ in a uniform $n$-element sorting network.
Recall the statement of Theorem \ref{T:main-3}:

\begin{customthm}{1.5}
Let $t \in [0, 1]$, and let $\eta_t$ be a subsequential limit of $\eta^n_t$. Then the support of the random measure $\eta_t$ is almost surely contained in the support of $\mathfrak{Arch}_t$.
\end{customthm}

\begin{proof}
Fix $t$, and suppose that $\eta_t$ is the distributional limit of the subsequence $\eta^{n_i}_t$. Since the sequence $Y_n$ is precompact by Lemma \ref{L:precompact}, there must be a subsubsequence $Y_{n_{i_k}}$ which converges in distribution to a random variable $Y$ in $\scrD$. Then the support of the random measure  $\eta_t$ is almost surely contained in the support of the law of $(Y(0), Y(t))$. Therefore we just need to check that  $(Y(0), Y(t)) \in \supp(\mathfrak{Arch}_t)$ almost surely.

\medskip

For $x \in [-1, 1]$, let $\prob_{x}$ be the conditional distribution of $Y$ given that $Y(0) = x$. By Theorem \ref{T:main}, for almost every $x \in [-1, 1]$, 
\begin{equation}
\label{E:circle}
\prob_{x}(Y\text{ is }\pi\sqrt{1-y^2}\text{-Lipschitz}) = 1.
\end{equation}

Now if $y(t)$ is a $\sqrt{1 - y^2}$-Lipschitz path with $y(0) = -y(1) = x$, then $y$ is bounded by the solutions of the initial value problems $f'(t) = \pm \sqrt{1 - f^2(t)} ; f(0) = x$ and $f'(t) = \pm \sqrt{1 - f^2(t)} ; f(1) = - x$. Therefore for any $t \in [0, 1]$, we have that
\begin{equation}
\label{E:two-ineq}
x \cos (\pi t) - \sqrt{1- x^2} \sin (\pi t) \le y(t) \le x \cos (\pi t) + \sqrt{1- x^2} \sin (\pi t).
\end{equation}
By using that $\mathfrak{Arch}_t \eqd (X, X \cos(\pi t) + Z \sin( \pi t))$, where $(X, Z) \eqd \mathfrak{Arch}_{1/2}$, and by using that the support of $\mathfrak{Arch}_{1/2}$ is the unit disk, the above inequality implies that $(x, y(t)) \in \supp(\mathfrak{Arch}_t).$ Combining this with \eqref{E:circle} completes the proof.
\end{proof}

Again, recall the statement of Theorem \ref{T:main-2}.

\begin{customthm}{1.6}
Suppose that $Y$ is a subsequential limit of $Y_n$. Then for any $\ep > 0$,
\begin{align*}
\prob \lf( Y(0) \ge 1 - \ep \mathand ||Y(t) - \cos(\pi t)||_u \ge \sqrt{2\ep}\rg) &= 0, \qquad \mathand \\
\prob \lf(Y(0) \le -1 + \ep \mathand ||Y(t) + \cos(\pi t)||_u \ge \sqrt{2\ep}\rg) &= 0.
\end{align*}
\end{customthm}

\begin{proof}

By Theorem \ref{T:main}, we have that almost surely
\begin{equation*}
Y(0) \cos (\pi t) - \sqrt{1- Y^2(0)} \sin (\pi t) \le Y(t) \le Y(0) \cos (\pi t) + \sqrt{1- Y^2(0)} \sin (\pi t).
\end{equation*}
for every $t$. This is simply \eqref{E:two-ineq} applied to $Y$. Elementary calculus gives that $||Y(t) - \cos(\pi t)||_u \le \sqrt{2(1 - Y(0))}$ and similarly shows that $||Y(t) + \cos(\pi t)||_u \le \sqrt{2(1 - Y(1))}$. 
\end{proof}

\section{Open problems}

The subsequent paper \cite{dauvergne3} proves Conjecture \ref{CJ:sine-curves}, Conjecture \ref{CJ:matrices}, and the other sorting network conjectures from \cite{angel2007random}. This gives a full description of the global limit of random sorting networks. In this section, we give a set of conjectures that focus on refining the understanding of convergence to this limit. Some of these conjectures are implicit in other papers or pictures, or have arisen in previous discussions but were not written down.

\medskip

Recall that $\sig_G^n$ is an $n$-element uniform random sorting network in the global scaling. Let $j \in \{1, \ddd n\}$, and consider the random complex-valued function
$$
Z^n_j(t) = e^{\pi i t} \lf[\sig^n_G(j, t) + i\sig^n_G(j, t + 1/2) \rg], \qquad t \in [0, 1/2].
$$
For a fixed $t$, $(Z^n_1(t), \ddd Z^n_n(t))$ is the set of points in the scaled permutation matrix for $\sig^n(\cdot, t + 1/2)(\sig^n)^{-1}(\cdot, t)$ after a counterclockwise rotation by $2\pi t$. The random vector-valued function $F^n(\cdot) = (Z^n_1(\cdot), \ddd, Z^n_n(\cdot))$ then gives a ``halfway permutation matrix evolution" for $\sig^n$ modulo uniform rotation (see Figure \ref{fig:window}). Conjecture \ref{CJ:sine-curves} implies that
$$
\max_{j \in [1, n]} \max_{s, t \in [0, 1]} |Z^n_j(t) - Z^n_j(s)| \cvgp 0 \qquad \mathas n \to \infty.
$$
Figure \ref{fig:window} suggests that the size of the fluctuations for each of the functions $Z^n_j$ is of order $n^{-1/2}$, and that the size is inversely proportional to the density of the Archimedean distribution at the point $Z^n_j(0)$. This leads to the first conjecture. 

\medskip

\begin{figure}
   \centering
   \includegraphics[scale= 1]{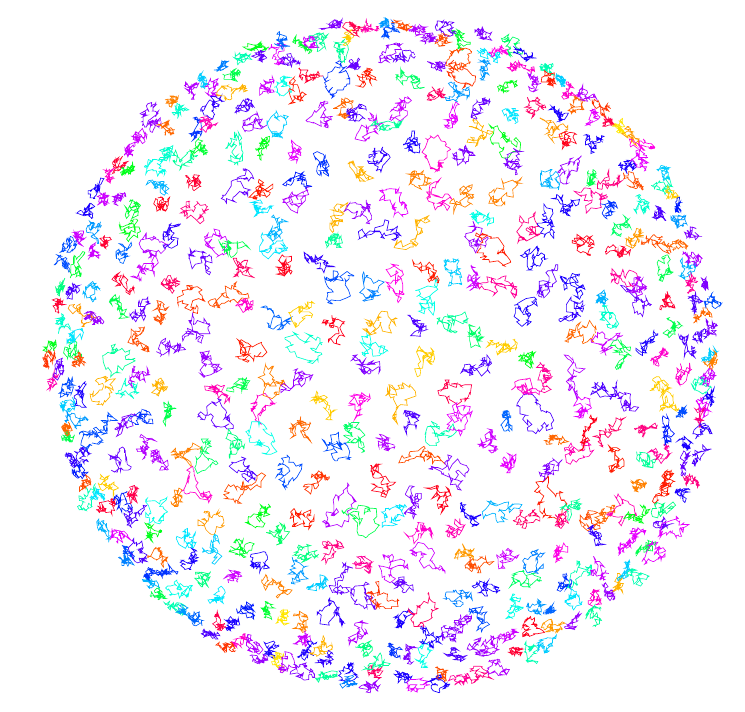}
   \caption{Images of the functions $Z^{500}_j(t)$. All paths are localized, and the distribution of these localized paths within $[-1, 1]^2$ is given by $\mathfrak{Arch}_{1/2}$. This figure originally appeared in \cite{angel2007random}.}
   \label{fig:window}
\end{figure}

\begin{conj}
\label{CJ:fluct}
Let $U$ be a uniform random variable on $[0, 1]$, independent of all the random sorting networks $\sig^n$. For each $n$, let $J_n$ be a uniform random variable on $\{1, \ddd n \}$, independent of $\sig^n$ and $U$.
\begin{enumerate}[nosep,label=(\roman*)]
\smallskip
\item The sequence of random variables $\{ n\Var(Z^n_{J_n}(U) \; | \; J_n, \sig^n) : n \in \nat\}$ is tight.
\item There exist independent random variables $X_1, X_2$ such that
\begin{equation*}
\lf(n\Var(Z^n_{J_n}(U) \; | \; J_n, \sig^n), |Z^n_{J_n}(0)| \rg) \cvgd (X_1\sqrt{1 - X_2^2}, X_2).
\end{equation*}
\end{enumerate}

\end{conj}

The second conjecture concerns the maximum value of the fluctuations. 

\begin{conj}
\label{CJ:fluct-max}
For any $\ep > 0$, 
\begin{equation*}
\max_{j \in [1, n]} \sup_{s, t \in [0, 1]} n^{1/2- \ep}|Z^n_{j}(t) - Z^n_j(s)| \to 0 \qquad \text{ in probability} \mathas n \to \infty.
\end{equation*}
\end{conj}

We now look at the local structure of the half-way permutation (see Figure \ref{fig:halfway}). Let $(x, y)$ be a point in the open unit disk, and consider the point process $\Pi_n(x, y) \sset \real^2$ given by
$$
\Pi_n(x, y) = \lf \{ \frac{\sqrt{n}}{\sqrt{2\pi}(1 - x^2 - y^2)^{1/4}}\lf(\sig^n_G(i, 0) - x,\sig^n_G(i, 1/2) - y\rg) : i \in \{1, \ddd n\} \rg\}.
$$
Heuristically, the $\sqrt{n}$ scaling, combined with the density factor of $\sqrt{2\pi}(1 - x^2 - y^2)^{1/4}$ from the Archimedean measure, should imply that for large $n$, the expected number of points of $\Pi_n(x, y)$ in a box $[a, b] \X [c, d] \sset \real^2$ is approximately $(b-a)(d-c)$.

\begin{conj}
\label{CJ:local-limit} 
There exists a rotationally symmetric, translation invariant point process $\Pi$ on $\real^2$ such that for any $(x, y)$ in the open unit disk, we have the following convergence in distribution:
$$
\Pi_n(x, y)  \cvgd \Pi.
$$
\end{conj}

We also consider deviations of the permutation matrix measures $\eta^n_t$ (see \eqref{E:eta-n-t} for the definition) from the Archimedean path $\{\mathfrak{Arch}_t : t \in [0, 1]\}$.
%


\begin{conj}
Let $U$ be any open set in the space of probability measures on $[-1, 1]^2$ with the topology of weak convergence, containing each of the measures $\mathfrak{Arch}_t$. There exist constants $c_1, c_2 > 0$ such that for all $n$,
$$
\prob \lf( \text{There exists } t \in [0, 1] \text{ such that } \eta^n_t \notin U \rg) \le c_1e^{-c_2n^2}.
$$
\end{conj}

Finally, Conjecture \ref{CJ:sine-curves} implies that if we know the location of particle $i$ after $\ep n^2$ swaps, then we know its trajectory. It is natural to ask to what extent this can be improved upon. 
The nature of the local limit suggests that the trajectory of particle $i$ should be determined after $O(n)$ steps.

\medskip

Again, let $J_n$ be a uniform random variable on $\{1, \ddd n \}$, independent of $\sig^n$. Let $S^n_t(i, \cdot)$ be the unique random curve of the form $A\sin(\pi t + \Theta)$ such that
$$
S^n_t(i, 0) =  \sig_G^n(i, 0) \;\; \mathand \;\; S^n_t(i, t) = \sig_G^n(i, t).
$$
\begin{conj}
For any $\ep > 0$, there exists a constant $C > 0$ such that 
$$
\liminf_{n \to \infty} \prob \lf( ||\sig^n_G(J_n, \cdot) - S^n_{C/n}(J_n, \cdot)||_u < \ep \rg) \ge 1 - \ep.
$$
\end{conj}

\subsection*{Acknowledgements} D.D. was supported by an NSERC CGS D scholarship. B.V. was supported by the 
Canada Research Chair program, the NSERC Discovery Accelerator
grant, the MTA Momentum Random Spectra research group, and the ERC 
consolidator grant 648017 (Abert). B.V. thanks Mustazee Rahman for several interesting and useful 
discussions about the topic of this paper.

\bibliographystyle{alpha}
\bibliography{LLRSNbib}

\end{document}